\documentclass[11pt]{amsart}
\usepackage{amsmath}
\usepackage{bbm}
\usepackage{mathrsfs}
\usepackage[english]{babel}
\usepackage[utf8x]{inputenc}
\usepackage[T1]{fontenc}
\usepackage{wrapfig}

\usepackage[top=3cm,bottom=2cm,left=3cm,right=3cm,marginparwidth=1.75cm]{geometry}

\usepackage{amsmath,amssymb}
\usepackage{graphicx}
\usepackage[colorinlistoftodos]{todonotes}
\usepackage[colorlinks=true, allcolors=blue]{hyperref}
\usepackage{cite}
\usepackage{enumitem} 
\usepackage[all,cmtip]{xy}
\usepackage{tikz-cd}
\usepackage{lmodern}
\usepackage{pgfplots}
\usetikzlibrary{intersections}
\usepackage{relsize}

\newtheorem{theorem}{Theorem}
\newtheorem{proposition}[theorem]{Proposition}
\newtheorem{corollary}[theorem]{Corollary}
\newtheorem{lemma}[theorem]{Lemma}
\newtheorem{definition}{Definition}

\theoremstyle{remark} 
\newtheorem{remark}[]{Remark}
\newtheorem{example}[]{Example}

\usepackage{mathtools} 
\mathtoolsset{showonlyrefs,showmanualtags}
\numberwithin{equation}{section}
\usepackage{amsmath}

\newcommand{\R}{\mathbb{R}}
\newcommand{\Qp}{\mathbb{Q}_p}
\newcommand{\G}{\mathbb{G}}

\newcommand{\C}{\mathbb{C}}
\newcommand{\RP}{\mathbb{R}\mathrm{P}}
\newcommand{\QP}{\mathbb{Q}_p\mathrm{P}}
\newcommand{\Zp}{\mathbb{Z}_p}
\newcommand{\Zn}{\mathbb{Z}/p^n\mathbb{Z}}
\newcommand{\CP}{\mathbb{C}\mathrm{P}}

\newcommand{\la}{\lambda}
\renewcommand{\epsilon}{\varepsilon}

\newcommand{\be}{\begin{equation}}
\newcommand{\ee}{\end{equation}}

\title[Probabilistic enumerative geometry over $p$--adic numbers]{Probabilistic enumerative geometry over $p$--adic numbers: linear spaces on complete intersections}

\author{Rida Ait El Manssour}
\address{MPI-MiS Leipzig, Inselstra{\ss}e 22, 04103 Leipzig, Germany.}
\email{rida.manssour@mis.mpg.de}

\author{Antonio Lerario}
\address{SISSA, Via Bonomea 265, 34136 Trieste, Italy}
\email{lerario@sissa.it}

\begin{document}

\maketitle
\begin{abstract}
We compute the expectation of the number of linear spaces on a random complete intersection in $p$--adic projective space. Here ``random'' means that the coefficients of the polynomials defining the complete intersections are sampled uniformly form the $p$--adic integers. We show that as the prime $p$ tends to infinity the expected number of linear spaces on a random complete intersection tends to $1$. In the case of the number of lines on a random cubic in three-space and on the intersection of two random quadrics in four-space, we give an explicit formula for this expectation.  
\end{abstract}

\section{Introduction}
\subsection{Classical and probabilistic enumerative geometry}
In this paper we deal with the enumerative problem of counting the number of $k$--dimensional projective subspaces (called ``$k$-flats'') on a complete intersection in $n$--dimensional $p$--adic projective space.

If this problem is approached over an algebraically closed field, it is a classic of enumerative geometry and we get a generic answer, e.g. there are $27$ lines on a generic cubic surface in $\CP^3$. At this point the word ``generic'' has its standard meaning from algebraic geometry: the generic object of a family has a property if this property is true for all the elements of the family except possibly for a proper algebraic subset of the family; below we will also exploit the measure--theoretic nature of this notion. Over a non--algebraically closed field, in general, we do not get such a generic answer and the number of solutions depends on the choice of the defining equations for the complete intersection. For instance, on a generic real cubic surface in $\RP^3$ the number of real lines can be either $3, 7, 15$ or $27$ (meaning that all possibilities occur for open sets in the space of cubics). Over the field of $p$--adic numbers, to our knowledge, it is not even clear what these generic possibilities are. In fact a cubic polynomial over $p$--adic fields needs to have a lot of variables ($22$) to ensure that the cubic surface has at least one line (see \cite[Theorem 1.3 and Theorem 1.4]{brandes_dietmann_2020}).

This is a generalization of the problem of counting the number of zeroes of a polynomial of one variable: for the generic polynomial (i.e. for a polynomial whose discriminant is nonzero) the number of zeroes over $\C$ equals the degree of the polynomial; over a non-algebraically closed field (e.g. $\R$ or $\Qp$) it depends on the coefficients of the polynomial.

This motivates a probabilistic approach to the problem: when there is not a single \emph{generic} answer, we can put a probability distribution on the space of polynomials and ask for the \emph{expectation} of the number of solutions to the enumerative question. Over $\C$ this approach gives back the classical generic count.

Clearly the expectation of the number of solutions, in the non--algebraically closed case, depends on the choice of the probability distribution, but there are some distributions which are especially interesting as they have a clear geometric meaning: these are the distributions on the space of polynomials which are invariant under the action of the group of isometries of the projective space by change of variables; following the notation for the real case, we call them \emph{invariant distributions} (we will clarify this notion below).
For these distributions there are no preferred points or directions in the projective space.

\begin{example}Let $\{\xi_\alpha\}_{|\alpha|=d}$ be a family of independent gaussian variables with:
\be\label{eq:gaussian} \xi_{\alpha_0\cdots\alpha_n}\sim N\left(0,\frac{d!}{\alpha_0!\cdots \alpha_n!}\right).\ee Consider the following real polynomial with random coefficients:
\be\label{eq:kostlan} f(x)=\sum_{|\alpha|=d}\xi_{\alpha}\cdot x_0^{\alpha_0}\cdots x_n^{\alpha_n}.\ee
Using \eqref{eq:kostlan} we can turn the space $\R[x_0, \ldots, x_n]_{(d)}$ of real homogeneous polynomials of degree $d$ into a gaussian space, i.e. a space with a gaussian probability distribution. This distribution is called the \emph{Kostlan} distribution. The scaling coefficient for the variances of the gaussian variables $\xi_\alpha$ is what makes the distribution invariant under isometries: for every orthogonal transformation $g\in O(n+1)$ the random polynomial $f\circ R$ obtained by a linear change of variables has the same distribution of $f$. In this way, if we denote by $Z(f_1, \ldots, f_\nu)\subset \RP^n$ the common zero set of the polynomials $f_1, \ldots, f_\nu$, if they are sampled at random as in \eqref{eq:kostlan}, we have the notion of random complete intersection.  The expected cardinality of $Z(f_1, \ldots, f_n)\subset \RP^n$, with each $f_j$ of degree $d_j$ and defined as in \eqref{eq:kostlan}, is $\sqrt{d_1\cdots d_n}$, see \cite{ShSm3}. The expected number of real lines on a random real cubic surface $Z(f)\subset \RP^3$ defined by picking $f$ as in \eqref{eq:kostlan} is $6\sqrt{2}-3$, see \cite{BLLP}.
\end{example}

\begin{remark}In the real case, the probabilistic approach goes back to Kac \cite{kac43}, who computed the expected number of real zeroes of a random polynomial with i.i.d. standard gaussian coefficients. The geometric point of view of invariant distributions was first adopted by Edelman, Kostlan, Shub and Smale \cite{EdelmanKostlan95, ShSm3, EKS}, for counting the expectation of the number of solutions of a system of random equations. The extension of this approach to questions in real enumerative geometry was initiated by the second named author of this paper together with B\"urgisser \cite{PSC} and with Basu, Lundberg and Peterson \cite{BLLP}. In the real gaussian case the invariant distributions were classified by Kostlan \cite{Kostlan95} and, in a recent work \cite{3armoniche}, the first named author of this paper together with Belotti and Meroni provided a closed formula for the expectation of the number of real lines on a real random cubic surface for all the possible invariant distributions. 
\end{remark}
\subsection{The $p$--adic case} Let us now move to the $p$--adic case. We begin by setting up the geometric framework for the enumerative problems of our interest. Given homogeneous polynomials $f_1, \ldots, f_\nu$ with coefficients in $\Zp$ and of degrees $d_1, \ldots, d_\nu\in \mathbb{N}$, we denote by $Z(f_1, \ldots, f_\nu)\subset \QP^n$ their common zero set in the $p$--adic projective space. If the list of degrees $d_1, \ldots, d_\nu$ satisfies
\be\label{eq:codim} \sum_{j=1}^{\nu}\binom{k+d_j}{d_j}=(k+1)(n-k),\ee
 then for the generic choice of the polynomials, the number of $k$-flats on $Z(f_1, \ldots, f_\nu)$ is finite, see \cite[Th\'eor\`eme 2.1]{debarremanivel}.
 
As we already observed, the number of solutions to our enumerative problem strongly depends on the coefficients of the polynomials. For instance, already in the case $\nu=1$, $k=0$, $d=3$ and $n=1$, the possibilities for the number of $p$--adic zeroes in $\QP^1$ of a cubic polynomial $f\in \Zp[x_0, x_1]_{(3)}$ with nonzero discriminant are $0, 1$ or $3$ (note the difference with the Real case). It is therefore natural to approach this problem from the random point of view, and ask for the expectation of this number.
 
\begin{remark} Before turning to the probabilistic side, let us briefly explain the meaning of the condition \eqref{eq:codim}. We denote by $\Qp\G_{k,n}$ the Grassmannian of $k$-flats in $\QP^n.$ Every element $\ell\in \Qp\G_{k,n}$ can be seen as the projectivization $\ell=P(L)$ of a vector space $L\simeq \Qp^{k+1}\subset \Qp^{n+1}.$ The dimension of $\Qp\G_{k,n}$ is $(k+1)(n-k)$ (the right hand side of \eqref{eq:codim}). We denote by $\tau_{k,n}$ the the tautological vector bundle on $\Qp\G_{k,n}$: 
\be \tau_{k,n}=\{(P(L), v)\in \Qp\G_{k,n}\times \Qp^{n+1}\,|\,v\in L\}.\ee
The dual of this bundle is denoted by $\tau_{k,n}^*$: the fiber of the dual bundle over a point $\ell=P(L)$ is the set of linear functions on $L.$ For every $d\in \mathbb{N}$ we denote by $\mathrm{Sym}^{(d)}(\tau_{k,n}^*)$ the $d$--th symmetric power of $\tau_{k,n}^*$: the fiber of $\mathrm{Sym}^{(d)}(\tau_{k,n}^*)$ over a point $\ell=P(L)$ is the set of homogeneous polynomial functions of degree $d$ on $L.$  Notice that every $f\in \Zp[x_0, \ldots, x_n]_{(d)}$ gives rise to a section $\sigma_f$ of $\mathrm{Sym}^{(d)}(\tau_{k,n}^*)$, defined by $\sigma_f(\ell)=f|_{L}$. 

Given the list of degrees $d_1, \ldots, d_\nu$ we can consider the following vector bundle, with corresponding section:
 \be \label{eq:bundles}
\begin{tikzcd}
\bigoplus_{j=1}^{\nu}\Qp^{\binom{d_j+k}{d_j}} \arrow[rr, hook] &  & {\bigoplus_{j=1}^{\nu}\mathrm{Sym}^{(d_j)}(\tau_{k,n}^*)} \arrow[dd]                   \\
                                                               &  &                                                                                          \\
                                                               &  & {\Qp\G_{k,n}} \arrow[uu, "\sigma_{f_1}\oplus \cdots \oplus \sigma_{f_\nu}"', bend right]
\end{tikzcd}\ee
The rank of this vector bundle is $\sum_{j=1}^{\nu}\binom{k+d_j}{d_j}$ (the left hand side of \eqref{eq:codim}). The zero locus of the section $\sigma_{f_1}\oplus \cdots \oplus \sigma_{f_\nu}$ consists of the set of $k$--flats which are contained in $Z(f_1, \ldots, f_\nu)$. In particular we see that, if \eqref{eq:codim} is verified, for the generic choice of $f_1, \ldots, f_\nu$, the corresponding section vanishes at finitely many points (i.e. the zero locus of the section is zero--dimensional, in the language of \cite{debarremanivel}).
\end{remark}
 
 We now move to the probabilistic framework. The first step is to endow $\Zp[x_0, \ldots, x_n]_{(d)}$ with a probability distribution. To this end, we endow $\Qp$ with the Haar measure $\lambda$, normalized such that $\lambda(\Zp)=1$. In this way the $p$--adic integers become a probability space, and we call the corresponding distribution the \emph{uniform} distribution. This is a special case of what is called a \emph{Gaussian} distribution on $p$--adic fields, see \cite{EMMT}. Mimicking \eqref{eq:kostlan}, we also turn the space of polynomials with coefficients in $\Zp$ into a probability space.

\begin{definition}We define a probability distribution on $\Zp[x_0, \ldots, x_n]_{(d)}$ by
\be\label{eq:random} f(x)=\sum_{|\alpha|=d}\xi_{\alpha}\cdot x_0^{\alpha_0}\cdots x_n^{\alpha_n},\ee
where now $\{\xi_\alpha\}_{|\alpha|=d}$ is a family of i.i.d. uniform elements in $\Zp$. We will call this distribution the \emph{uniform} distribution.
\end{definition}
The uniform distribution on the space of polynomials is invariant under the action of $\mathrm{GL}_{n+1}(\Zp)$ by change of variables (Proposition \ref{propo:invariant}) and it has therefore a clear geometric meaning. Notice that proper algebraic sets in $\Zp[x_0, \ldots, x_n]_{(d)}$ have probability zero: for instance, with probability one the zero set of a random polynomial is  smooth.
The natural map $f\mapsto \sigma_f$ induces a probability distribution on the space of sections of the bundle \eqref{eq:bundles}. The zeroes  of $\sigma_{f_1}\oplus \cdots \oplus \sigma_{f_\nu}$ are nondegenerate with probability one.

Our first theorem computes the expectation (below denoted by ``$\mathbb{E}$'') of the number of zeroes of the random section $\sigma_{f_1}\oplus\cdots \oplus \sigma_{f_\nu}$, under the assumption \eqref{eq:codim}, i.e. the expectation of the number of $k$-flats on a random complete intersection $Z(f_1, \ldots, f_\nu)\subset \QP^n.$

\begin{theorem}\label{thm:complete}
Let $f_1, \ldots, f_\nu$ be independent random polynomials of degrees $d_1, \ldots, d_\nu$ sampled from the uniform distribution \eqref{eq:random}. Then
\be \lim_{p\to \infty}\mathbb{E}\#\{\textrm{$k$-flats on $Z(f_1, \ldots, f_\nu)\subset \QP^n$}\}=1.\ee
\end{theorem}

\begin{remark}Theorem \ref{thm:complete} is in sharp contrast with its real analogue, at least for the case of hypersurfaces. Denoting by $C_n$ the number of lines on a generic hypersurface of degree $2n-3$ in $\CP^n$ and by $E_n$ the expectation of the number of real lines on a random, Kostlan distributed real hypersurface of degree $2n-3$ in $\RP^n$, \cite[Theorem 12]{BLLP} states that:
\be \lim_{n\to \infty}\frac{\log E_n}{\log C_n}=\frac{1}{2}.\ee
Since $C_n$ grows super-exponentially in $n$ (in fact $\log C_n=2n\log n+O(n)$), on hypersurfaces of degree $2n-3$ in large dimensional projective spaces we expect to see many real lines -- there is in fact a deterministic lower bound for this number, see Section \ref{sec:signed} below. Here, at least for large $p$, as a consequence of Theorem \ref{thm:complete}, having many $p$--adic lines is an extremely rare event. In the $p$--adic case, we will also prove an asymptotic upper bound on our expectation as $n\to \infty$, see Theorem \ref{thm:n} below.
\end{remark}

\begin{remark}The probabilistic approach in the $p$--adic case has been introduced by \cite{evans}, who first computed the expectation of the number of zeroes of a system of random equations (with respect to a probability distribution which is different from \eqref{eq:random}). In the recent paper \cite{pIGF}, the second named author of the current paper, together with Kulkarni, proved a generalization of the integral geometry formula to the $p$--adic setting, allowing to deduce some results on random systems of equations distributed as in \eqref{eq:random}. Independently, the distribution of zeroes of a random uniform univariate polynomial was also recently studied by Caruso in \cite{Caruso}. We notice that \cite[Corollary 41]{pIGF} corresponds to the case $k=0$, $\nu=n$ of Theorem \ref{thm:complete}, in which case we can actually prove the following result.
\end{remark}
\begin{theorem}\label{thm:complete0}
Let $f_1, \ldots, f_n$ be independent random polynomials of degrees $d_1, \ldots, d_n$ sampled from the uniform distribution \eqref{eq:random}. Then
\be \mathbb{E}\# Z(f_1, \ldots, f_n)=1.
\ee
\end{theorem}
As a corollary of this result we compute the expectation of the absolute value of the determinant of a matrix $M_n\in \Zp^{n \times n}$ filled with i.i.d. uniform elements in $\Zp$ (Corollary \ref{coro:matrix}):
\be\label{eq:M} \mathbb{E}\{|\det M_n|_p\}=\frac{(p-1)p^n}{p^{n+1}-1}.\ee

Another case of special interest is the count of the number of lines on a cubic surface in $\QP^3,$ for which we can prove the following sharper version of Theorem \ref{thm:complete} (we also prove a similar result for the intersection of two random quadrics in $\QP^4$, see Theorem \ref{thm:quadrics} below).
\begin{theorem}\label{thm:cubic}The expected number of $p$--adic lines on a random uniform $p$--adic cubic surface in $\QP^3$ is $\frac{(p^3-1)(p^2+1)}{p^5-1}.$
\end{theorem}
The proof of Theorem \ref{thm:complete} is based on a $p$--adic version of the Kac--Rice formula for sections of vector bundles (Theorem \ref{thm:reduction}), which reduces the computation of the expectation of the number of zeroes of a random section to the expectation of the valuation of the determinant of a special random matrix $J$:
\be \label{eq:JJ}\mathbb{E}\#\{\textrm{$k$-flats on $Z(f_1, \ldots, f_\nu)\subset \QP^n$}\}=\mu(\Qp\G_{k,n})\cdot \mathbb{E}\left\{|\det(J)|_p\right\}.\ee
Here the matrix $J$ is a random square matrix with $(n-k)(k+1)$ columns and whose structure depends on $\nu, k$ and $d_1, \ldots, d_\nu$. For instance, when $k=0$ the matrix $J=M_n$ is the above matrix filled with i.i.d. uniform elements in $\Zp$ (which explains \eqref{eq:M}), but the general case is more complicated.  Theorem \ref{thm:complete} is based on the asymptotic analysis of \eqref{eq:JJ}; Theorem \ref{thm:cubic} and Theorem \ref{thm:quadrics} require instead a more delicate study, based on a counting argument in the reduction modulo $p^m$.
\subsection{Signed counts}\label{sec:signed}It is interesting to observe that over the Reals there is a deterministic lower bound on the number of real lines on a generic cubic surface. In fact, generalizing a construction of Segre \cite{Segre}, Okonek and Teleman \cite{OkTel} and Kharlamov and Finashin \cite{finkh} introduced a way to assign a sign to each line on the zero set of a hypersurface of degree $2n-3$ in $\RP^n$. The signed count of the number of real lines does not depend on the hypersurface and gives $(2n-3)!!$. In particular there are always $3$ lines on a smooth real cubic surface, and exactly $3$ if counted with signs. This enriched count has been extended to any field \cite{kass2017arithmetic}, and it gives an invariant with values in the Grothendick-Witt group of the field. In the $p$--adic case however this is not enough to guarantee a solution: there are $p$--adic cubic surfaces with no lines on them. 

In the real case the computation of the signed count can also be done in a probabilistic way, as showed in \cite{BLLP}. Let us explain this point in the simplest case of cubics. Using the Real version of the Kac--Rice formula, it can be proved that the expectation $E_3$ of the number of real lines on a random real cubic surface equals:
\be\label{eq:E3} E_3=\frac{1}{12\pi^2}\cdot \mu(\mathbb{R}\mathbb{G}(1,3))\cdot \mathbb{E}|\det J|\ee
where $\mu(\mathbb{R}\mathbb{G}(1,3))=2\pi^2$ is the volume of the real Grassmannian, and $J$ is the following matrix, filled with the random gaussian variables defined in \eqref{eq:gaussian}:
\be\label{eq:J} J= 
\begin{pmatrix}
\xi_{2010} &0 &\xi_{2001}  &0 \\ 
\xi_{1110} &\xi_{2010}  &\xi_{1101}  &\xi_{2001} \\ 
\xi_{0210} &\xi_{1110} &\xi_{0201}  &\xi_{1101} \\ 
0 &\xi_{0210}  &0  &\xi_{0201}
\end{pmatrix}.\ee
Note that this matrix is the Jacobian, in local coordinates, of the random section $\sigma_f$ at $\ell_0=P(L_0)$, where $L_0=\{x_2=x_3=0\}\subset \R^4.$ 
As we already mentioned, the Kac--Rice approach works also over the $p$--adics, and it is our starting point for the proofs of Theorem \ref{thm:complete} and Theorem \ref{thm:cubic}. For instance, in the case of cubics:
\be \mathbb{E}\# \{\textrm{$p$--adic lines on a random cubic surface}\}=\mu(\Qp\G(1,3))\cdot \mathbb{E}|\det J|_p\ee
where now $\mu(\Qp\G(1,3))$ denotes the volume of the $p$--adic Grassmannian (defined in Section \ref{sec:mms} below) and $J$ is the same matrix as in \eqref{eq:J}, this time filled with uniform variables in $\Zp.$ (The constant $\frac{1}{12\pi^2}$ appearing in \eqref{eq:E3} is the value at zero of the density of the coefficients of our random cubic polynomial; in the $p$--adic case this constant is $1$.)

The interesting point now is that in the real case the signed count can be obtained simply by computing the expectation of the quantity on the r.h.s. of \eqref{eq:E3}, but where we remove the modulus from the determinant of $J$ see \cite[Proposition 3]{BLLP}. (A similar statement holds true for higher dimensions/degrees). Mimicking the proof of \cite[Proposition 2]{BLLP} we can compute the expectation of the determinant of $J$ in the $p$--adic case. Computing this expectation amount now to compute the integral of a function with values in $\Qp$. We do this using the notion of \emph{Volkenborn integral}, which for a continuous function $f\, : \Zp^k \to \Qp$ is defined by the following limit (if it exists):
$$\int_{\Zp^k}f(x)\textrm{dx}=\lim_{n\rightarrow +\infty} \frac{1}{p^{kn}}\sum_{a_1=0}^{p^n-1}\cdots \sum_{a_k=0}^{p^n-1}f(a_1,\ldots,a_k).$$
In the case of cubics, this would give:
\begin{align*}
 \mathbb{E}\{\det(J)\}&=\mathbb{E}\bigg\{(\xi_{2010}\xi_{0201}-\xi_{0210}\xi_{2001})^2-(\xi_{2010}\xi_{1101}-\xi_{1110}\xi_{2001})(\xi_{1110}\xi_{0201}-\xi_{0210}\xi_{1101})\bigg\}\\
 &=\int_{\Zp^6}   (x_1x_6-x_3x_4)^2-(x_1x_5-x_2x_4)(x_2x_6-x_3x_5) \textrm{d$\xi_1\ldots$d$\xi_6$}\\
 &= \lim_{n\rightarrow +\infty}\frac{1}{p^{6n}}\sum_{0\le a_1,\ldots,a_6\le p^n-1}(a_1a_6-a_3a_4)^2-(a_1a_5-a_2a_4)(a_2a_6-a_3a_5)\\
 &=\lim_{n\rightarrow +\infty}\frac{1}{p^{6n}}\cdot\frac{p^{6n}(p^n-1)^2(5p^{2n}+p^{n}-4)}{36}\\
 &=-\frac{1}{9}\in \Qp.
\end{align*}
We do not have a clear interpretation of the meaning of this number, but we believe this can be the starting point for future investigations.

\subsection*{Acknowledgements}The authors wish to thank Yassine El Maazouz and Avinash Kulkarni for stimulating discussions and Bernd Sturmfels for his constant support. The first named author was supported by SISSA (Trieste).

\section{Preliminaries}
\subsection{$p$--adic spaces as metric measure spaces}\label{sec:mms}
Let $p$ be a prime number and denote by $|\cdot|_p$ the $p$--adic absolute value. We endow the space $\Qp^n$ with the norm $\|\cdot\|_p$ defined by:
\be\forall (a_1, \ldots, a_n)\in \Qp^n:\quad \|(a_1, \ldots, a_n)\|_p:=\sup_{i}|a_i|_p .\ee
In this way $\Qp^n$ becomes a metric space. We endow it also with the Haar measure $\la$ normalized over the unit ball $\Zp^n$. i.e. $\la(\Zp^n)=1$. In particular, for any ball $B(x, p^{-m})\subset \Qp^n$ we have:
\be \la(B(x, p^{-m}))=p^{-nm}.\ee
The measure $\lambda$ is invariant under the group $\mathrm{GL}_n(\Zp)$, i.e. for every measurable set $U\subseteq \Qp^n$ and every element $M\in \mathrm{GL}_n(\Zp)$ we have $\lambda(M(U))=\lambda(U).$ 

Once restricted to $\Zp^n$, the measure $\lambda$ becomes a probability measure. We call the corresponding probability distribution on $\Zp^n$ the \emph{uniform distribution}.  Using this terminology, we can restate the $\mathrm{GL}_n(\Zp)$--invariance of $\lambda$ as follows: let $\xi$ be a uniformly distributed random vector in $\Zp^n$; then for every $M\in \mathrm{GL}_n(\Zp)$ the vectors $\xi$ and $M\xi$ have the same distribution.

We view the set of $n\times n$ matrices as a subset of $\Qp^{n\times n}$, with the corresponding measure. Under this identification, the set $\mathrm{GL}_n(\Zp)$ is contained in $\Zp^{n\times n}$ and it consists of matrices which are invertible modulo $p$; its measure is given by \cite[Theorem 4.1]{EVANS200289}:
\be\label{eq:volGL}\la(\textrm{GL}_n(\Zp))=\left(1-\tfrac{1}{p}\right)\cdots\left(1-\tfrac{1}{p^n}\right).\ee
More generally, again by \cite{EVANS200289}, we have:
 \be\label{eq:12}\la(\{ |\det(M)|_p=p^{m}\})=\Pi_n p^{-m}\frac{\Pi_{n+m-1}}{\Pi_m \cdot \Pi_{n-1}} \ee
 where $\Pi_k=\left(1-\tfrac{1}{p}\right)\cdots\left(1-\tfrac{1}{p^k}\right)$.
 
 We endow the Grassmannian $\Qp\G_{k,n}$ with the normalized\footnote{In this way the measure of the Grassmannian equals $\mu(\Qp\G_{k,n})=\frac{\lambda(\mathrm{GL}_{n+1}(\Zp))}{\lambda(\mathrm{GL}_{n-k}(\Zp))\lambda(\mathrm{GL}_{k+1}(\Zp))}$, reflecting the Real case.} pushforward measure 
 \be \mu:=\frac{1}{\lambda(\mathrm{GL}_{n-k}(\Zp))\lambda(\mathrm{GL}_{k+1}(\Zp))}q_*\la,\ee
 where the map $q: \mathrm{GL}_{n+1}(\mathbb{Z}_p) \to  \Qp\G_{k,n}$ is given by
\be q:M \mapsto \mathrm{im}\left(M\begin{pmatrix}
    \mathbf{1}_{k+1}\\
    0
    \end{pmatrix}\right).
\ee
We refer the reader to \cite{pIGF} for more details on the properties of $\Qp^n$ and related spaces as metric measure spaces. On these spaces continuous functions can be integrated and a useful proposition for the sequel is the following change of variable formula from \cite{evans}.
 
 \begin{proposition}\label{Kac}
  Let $X$ be an open subset of $\Qp^m$ and $f\, :X\to \Qp^n$ a continuously differentiable map. For all measurable subset $Y$ of $\Qp^n$ we have: 
  $$\int_Y \#\{f=y\}\textrm{dy}=\int_{f^{-1}(Y)}|\det(Jf(x))|_p\textrm{dx}.$$
\end{proposition}

Next lemma, which is going to  be used in the proofs of Theorem \ref{thm:cubic} and Theorem \ref{thm:quadrics}, gives the analogue of Riemann sums for performing integrals of continuous functions on $\Zp^n.$
\begin{lemma}\label{lemma:integral}
Let $f: \Zp^n \to \R$ be a continuous function. Then $f$ is integrable and its integral is given by:
\be 
\int_{\Zp^n} f(x)\textrm{dx}=\lim_{m\rightarrow +\infty} \frac{1}{p^{mn}}\sum_{0\le a_1,\ldots,a_n\le p^m-1} f(a_1,\ldots, a_n)
\ee
\end{lemma}
\begin{proof}The proof is elementary.
Since $f$ is continuous, it is bounded and therefore integrable. For every $m\in \mathbb{N}$, consider the function:
\be f_m(x):=\sum_{0\le a_1,\ldots,a_n\le p^m-1} f(a_1,\ldots, a_n)\cdot \chi_{B(a;p^{-m})}(x),\ee
where $a=(a_1, \ldots, a_n)$ and $\chi_{B(a;p^{-m})}$ is the indicator function of $B(a;p^{-m})$. The sequence of functions $\{f_m\}_{m\in \mathbb{N}}$ converge pointwise to $f$ and is dominated by the constant $\max_{x\in \Zp^n}|f(x)|$.  Therefore:
\begin{align} \int_{\Zp^n}f(x)dx&=\int_{\Zp^n}\lim_{m\to \infty}f_m(x)dx\\
&=\lim_{m\to \infty}\int_{\Zp^n}f_m(x)dx\\
&=\lim_{m\rightarrow +\infty} \lambda(B(a, p^{-m}))\sum_{0\le a_1,\ldots,a_n\le p^m-1} f(a_1,\ldots, a_n)\\
&=\lim_{m\rightarrow +\infty} \frac{1}{p^{mn}}\sum_{0\le a_1,\ldots,a_n\le p^m-1} f(a_1,\ldots, a_n).
\end{align}
\end{proof}
Let us also recall the following proposition, which is going to be used in the proofs of Theorem \ref{thm:cubic} and Theorem \ref{thm:complete}; it is a consequence of the $p$--adic singular value decomposition, see \cite[Theorem 3.1]{EVANS200289}.
\begin{proposition}\label{SVD}
  Let $M\in (\Zn)^{k\times k}$ with $\det(M) \ne 0$, then there exist unique integers $0\le u_1\le \cdots \le u_k\le n-1$ and $U,V \in GL_k(\Zn)$ such that:
  $$M=U\begin{pmatrix}
  p^{u_1} &  &  \\
  & \ddots & \\
  & & p^{u_k}
  \end{pmatrix}V.$$
\end{proposition}
\subsection{The reduction modulo $p^m$}
For every $m\in \mathbb{N}$, let us denote by
\be \pi_m\, : \Zp^k\to (\mathbb{Z}/p^m\mathbb{Z})^k\ee
the map that sends a vector to its reduction modulo $p^m$. This map is a ring homomorphism and for a set $U\subseteq \Zp^n$ we denote by $N_m(U)$ the cardinality of $\pi_m(U).$ The following result from \cite{pIGF} relates the volume of a set with the cardinality of its reduction modulo $p^m$. 
\begin{lemma}[Lemma 25 from \cite{pIGF}]\label{Cardinal}Let $U\subseteq A\subseteq \Zp^n$ be an open and compact subset of an algebraic set $A$. Then $N_m(U)$ equals the minimum number of affine balls of radius $p^{-m}$ that we need to cover $U$. In particular, if $U\subseteq \Zp^n$ is an open set, then
\be \lambda(U)=p^{-mn}N_m(U).\ee
\end{lemma}
In particular, since $\pi_m(\mathrm{GL}_n(\Zp))=\mathrm{GL}_n(\Zp/p^m\Zp)$, using \eqref{eq:volGL} we see that:
\be \#\mathrm{GL}_n(\Zp/p^m\Zp)=p^{mn^2}\lambda(\mathrm{GL}_n(\Zp))=p^{mn^2}\left(1-\tfrac{1}{p}\right)\cdots\left(1-\tfrac{1}{p^n}\right).\ee
More generally, for $\ell < m$, we have:
\be\label{eq:Cardinal2} \#\left\{M\in (\mathbb{Z}/p^m\mathbb{Z})^{n\times n}\, \bigg|\, |\det(M)|_p=p^{-\ell}\right\}=p^{mn^2}\lambda\left(\left\{M\in \Zp^{n\times n}\, \bigg|\, |\det(M)|_p=p^{-\ell}\right\}\right).\ee

\subsection{Random $p$--adic polynomials}
Let $\Qp[x_0,\ldots,x_n]_{(d)}$ be the space of homogeneous polynomials of degree $d$, $n+1$ variables and coefficients in $\Qp;$ denote by $D_{d, n}=\binom{d+n}{d}$ its dimension.
\begin{definition}
We define a random polynomial $f$ in $\Zp[x_0,\ldots,x_n]_{(d)}$ as follows: 
\be\label{eq:rp} f=\sum_{|\alpha|=d}\xi_\alpha x_0^{\alpha_0}\cdots x_n^{\alpha_n}\ee
where $\xi_{\alpha}$ are independent random variables uniformly distributed in $\Zp$. We call the probability distribution induced by \eqref{eq:rp} on $\Zp[x_0, \ldots, x_n]_{(d)}$ the \emph{uniform} distribution.
\end{definition}
Identifying a polynomial in $\Zp[x_0, \ldots, x_n]_{(d)}$ with the list of its coefficients (in the monomial basis), we see that the uniform distribution \eqref{eq:rp} coincides with the uniform distribution on the unit ball $\Zp^{D_{d,n}}$ defined in Section \ref{sec:mms}.
\begin{proposition}\label{propo:invariant}
  The uniform probability distribution on $\Zp[x_0, \ldots, x_n]_{(d)}$ is invariant under the action of $\textrm{GL}_{n+1}(\Zp)$ by change of variables. 
\end{proposition}
\begin{proof}
Let us denote by $\rho:\mathrm{GL}_{n+1}(\Zp)\to \mathrm{GL}_{D_{d,n}}(\Qp)=\mathrm{GL}(\Qp[x_0, \ldots, x_n]_{(d)}) $ the representation by change of variables. Observe that for a matrix $M\in \mathrm{GL}_{n+1}(\Zp)$, the matrix $\rho(M)$ has coefficients in $\Zp$ and, since $\rho(M)^{-1}=\rho(M^{-1})$, its inverse also has coefficients in $\Zp$. It follows that $\rho(M)\in\mathrm{GL}_{D_{d,n}}(\Zp)$. The uniform probability distribution on $\Zp^{D_{d,n}}$ is invariant under elements in $\mathrm{GL}_{D_{d,n}}(\Zp)$ and therefore it is in particular invariant under change of variables in $\textrm{GL}_{n+1}(\Zp)$.
\end{proof}

\subsection{Counting $k$-flats as zeroes of sections}We are going now to perform some preliminary reductions for our problem of counting the expectation of the number of $k$-flats on a random complete intersection.

Let $U\subset \Qp\G_{k,n}$ be the open set consisting of all projective spaces $\ell=P(L)$ whose first entry of the Pl\"ucker coordinates is nonzero and denote by $\phi:U\to \Qp^{(n-k)\times (k+1)}$ the map:
\be \phi(\ell)=\phi\left(\mathrm{im}\begin{pmatrix}
    \mathbf{1}_{k+1} \\
    A
    \end{pmatrix}\right)=A.
    \ee
This map is well defined and $(U,\phi)$ is a chart of the $p$--adic manifold $\Qp\G_{k,n}$.

\begin{lemma}\label{lemma:disj}Let $U_m=\phi^{-1}(B(0, p^{-m})).$ There exists $g_1, \ldots, g_{N_m}\in \mathrm{GL}_{n+1}(\Zp)$ such that the Grassmannian can be written as a disjoint union:
\be \Qp\G_{k,n}=\bigcup_{i=1}^{N_m}(g_i\cdot U_m).\ee
Moreover:
\be\label{eq:Um} \mu(U_m)=p^{-mn^2}\quad \textrm{and}\quad N_m=\mu(\Qp\G_{k,n})\cdot\frac{1}{p^{mn^2}}.\ee
\end{lemma}
\begin{proof}First observe that for every $g\in \mathrm{GL}_{n+1}(\Zp)$ the set $g\cdot U_m$ is open in $\Qp\G_{k,n}$ and $\{g\cdot U_m\}_{g\in \mathrm{GL}_{n+1}(\Zp)}$ is an open cover of the Grassmannian. By compactness we can extract a finite subcover, and it remains to prove that the open sets from this finite cover can be taken to be disjoint.

Let us consider the following subgroup of  $\mathrm{GL}_{n+1}(\Zp)$:
\begin{align}\mathcal{H}_m:=\bigg\{\begin{bmatrix}
A &B\\
C&D
\end{bmatrix} \in \mathrm{GL}_{n+1}(\mathbb{Z}_p) \, \bigg|&\, A \in GL_{k+1}(\mathbb{Z}_p),\, D\in GL_{n-k}(\mathbb{Z}_p),\\
&C\in B(0;p^{-m}),\, B\in \mathbb{Z}_p^{(k+1) (n-k)} \bigg\}.\end{align}
We will show that given $g\in \textrm{GL}_{n+1}(\Zp)$ and $X\in U_m$, then $g\cdot X \in U_m$ if and only if $g\in \mathcal{H}_m$. Moreover, $\forall \, g_1,\,g_2 \in \textrm{GL}_{n+1}(\Zp)$: either $g_1\cdot U_m=g_2 \cdot U_m$ or $g_1\cdot U_m\cap g_2\cdot U_m=\emptyset$. From this it follows that the previous cover of $\Qp\G_{k,n}$ can be taken to be disjoint. 

Set $X=\begin{bmatrix}
 \mathbf{1}_{k+1}\\
 \hat{X}
 \end{bmatrix}$ where $\hat{X}\in B(0;p^{-m})$ and $g=\begin{bmatrix}
A &B\\
C&D
\end{bmatrix} \in \textrm{GL}_{n+1}(\Zp)$. With this notation we have
 \[g\cdot X=\begin{bmatrix} A+B\hat{X}\\
 C+D\hat{X}\end{bmatrix}\in U_m \Leftrightarrow  \begin{cases} \det(A+B\hat{X}) \ne 0\\
 (C+D\hat{X})(A+B\hat{X})^{-1} \in B(0,p^{-m})
 \end{cases}.
 \]
 Suppose first that $g\cdot X=\begin{bmatrix} A+B\hat{X}\\
 C+D\hat{X}\end{bmatrix}\in U_m$. Since $A+B\hat{X}\in \mathbb{Z}_p^{(k+1)\times(k+1)}$, then
 \[ \left | C+D\hat{X} \right |_p\leq \left | (C+D\hat{X})(A+B\hat{X})^{-1} \right |_p.\]
 It follows that $C+D\hat{X}\in B(0;p^{-m})$, which implies $ C\in B(0;p^{-m})$. By reducing $g$ modulo $p$, we get that $$\det(g)=\det(A)\cdot \det(D) \ \ (mod\ \ p)$$ Therefore $A \in \textrm{GL}_{k+1}(\Zp)$, and $D \in \textrm{GL}_{n-k}(\Zp).$
 The converse is trivial.
 
 Suppose now that $g_1\cdot U_m\cap g_2\cdot U_m\ne \emptyset$ i.e. $g_1^{-1}g_2\cdot U_m\cap U_m \ne \emptyset$. By the previous  point $g_1^{-1}g_2\in \mathcal{H}_m$ which implies $g_1^{-1}g_2\cdot U_m\subseteq U_m$. But we also have $g_2^{-1}g_1  \in \mathcal{H}_m$, and therefore $g_1^{-1}g_2\cdot U_m=U_m$ i.e. $g_1\cdot U_m=g_2\cdot U_m$.
 
 The properties \eqref{eq:Um} follow immediately from the definition of the pushforwad measure $\mu$ and the structure of $\mathcal{H}_m$.
\end{proof}

Recall now that for every $d\in \mathbb{N}$ we have denoted by $\mathrm{Sym}^{(d)}(\tau_{k,n}^*)$ the vector bundle which is the $d$-th symmetric power of the dual of the tautological bundle on $\Qp\G_{k,n}$: the fiber over a point $\ell=P(L)$ is the set of homogeneous polynomials of degree $d$ on $L\simeq \Qp^{k+1}.$ Given the list of degrees $d_1, \ldots, d_\nu$ satisfying \eqref{eq:codim}, we define the vector bundle:
\be E=\bigoplus_{j=1}^{\nu}\mathrm{Sym}^{(d_j)}(\tau_{k,n}^*)\to\Qp\G_{k,n} \ee

Over the open set $U$ we  have a trivialization of the vector bundle $E$: 
\be 
\begin{tikzcd}
{E|_U} \arrow[rdd, "\pi"] \arrow[rr, "h"] &   &  {U\times\left(\bigoplus_{j=1}^{\nu} \Qp^{\binom{d_j+k}{d_j}}\right)} \arrow[ldd, "p_1"] \\ &   &     \\
 & U &                                    \end{tikzcd} \ee
 
Given the list of polynomials $f_1, \ldots, f_\nu$ with degrees respectively $d_1,\ldots, d_\nu$ that satisfy \eqref{eq:codim}, we get a section $\sigma_f:\Qp\G_{k,n}\to E$ of the form \be\sigma_f=\sigma_{f_1}\oplus\cdots\oplus \sigma_{f_{\nu}}.\ee
If the polynomials $f_1, \ldots, f_\nu$ are random, we use the section $\sigma_f$ to define the random map
\be\label{eq:psif}\psi_f:=p_2\circ h\circ \sigma_f|_U \circ \phi^{-1} :\Qp^{(n-k)(k+1)} \to \mathbb{Q}_p^{(n-k)(k+1)}.\ee
This map takes a matrix $A\in \Qp^{(n-k)\times (k+1)}$ and gives the list of coefficients of the restriction of the polynomials $f_1, \ldots, f_\nu$ to the subspace $\phi^{-1}(A).$ These coefficients clearly depend on the choice of the trivialization $h$ and we choose such trivialization in such a way that the components of the random map $\psi_f=(\psi_{f_1}, \ldots, \psi_{f_{\nu}})$ are the coefficients of the polynomials
\be\label{eq:trivi} f_j\left(\begin{pmatrix}
    \mathbf{1}_{k+1} \\
    A
    \end{pmatrix}y\right)\in \Qp[y_0, \ldots, y_k]_{(d_j)}, \quad y=(y_0, \ldots, y_k).\ee
\begin{lemma}\label{lemma:ecco}Using the above notations, for every $m\in \mathbb{N}$ we have \be\mathbb{E}\#\{\sigma_f=0\}=N_m\cdot \mathbb{E}\#\{\psi_f|_{B(0, p^{-m})}=0\}.\ee
\end{lemma}
\begin{proof}Using Lemma \ref{lemma:disj}, we can cover the Grassmannian with disjoint open sets $\{g_iU_m\}_{i=1}^{N_m}$ and therefore:
\begin{align}\mathbb{E}\#\{\sigma_f=0\}&=\sum_{i=1}^{N_m}\mathbb{E}\#\{\sigma_f=0\}\cap g_iU_m\\
&=\sum_{i=1}^{N_m}\mathbb{E}\#\{\sigma_f\circ g_i=0\}\cap U_m\\
&=\sum_{i=1}^{N_m}\mathbb{E}\#\{\sigma_f=0\}\cap U_m\\
&=N_m \cdot \mathbb{E}\#\{\sigma_f=0\}\cap U_m,
\end{align}
where from the second to the third line we have used the $\mathrm{GL}_{n+1}(\Zp)$--invariance. On the other hand $\sigma_f$ vanishes at $\ell\in U_m$ if and only if $\psi_f$ vanishes at $\phi(\ell)\in B(0, p^{-m})$ and the conclusion follows.
\end{proof}
Recall that the total volume of the Grassmannian is given by:
\be \mu(\Qp\G_{k,n})= \frac{\la(\mathrm{GL}_{n+1}(\Zp))}{\la(\mathrm{GL}_{k+1}(\Zp))\la(\mathrm{GL}_{n-k}(\Zp))}.\ee
Next result is our version of the $p$--adic Kac-Rice formula for sections of vector bundles: it allows to reduce the expectation of the number of zeroes of $\sigma_f$ to the evaluation of a random determinant. 
\begin{theorem}\label{thm:reduction}
Let $f_1, \ldots, f_\nu$ be random polynomial as before. The expected number of $k$-flats on $Z(f_1, \ldots, f_\nu)$ is given by:
\begin{equation*}
    \mathbb{E}\#\left\{\sigma_f=0\right\}=\mu(\Qp\G_{k,n})\cdot \mathbb{E}\left\{|\det(J\psi_f(0))|_p\right\}.
\end{equation*}
\end{theorem}
\begin{proof}
Let $D=(k+1)(n-k).$ For every $a=(a_1,\ldots, a_\nu) \in \bigoplus_{j=1}^\nu\Zp[x_0,\ldots,x_k]_{(d_j)}\simeq \Zp^D$ denote by $f_a=(f_{a_1}, \ldots, f_{a_\nu})$ the list of polynomials in $n+1$ variables and of degrees $(d_1, \ldots, d_\nu)$ obtained as follows. For every $i=1, \ldots, \nu$ consider the polynomial $g_{a_i}\in \Qp[y_0, \ldots, y_k]_{(d_j)}$ whose coefficents are the entries of the vector $a_i$, and then set:
\be f_a(x_0, \ldots, x_n):=g_a(x_0, \ldots, x_k).\ee
(So $f_{a_i}$ is a polynomial on $\Qp^{n+1}$ which depends only on the first $k+1$ variables).

If $\xi\in \Zp^D$ denotes the list of coefficients of the random vector $f$, the list of the coefficients of $f-f_a$ is $\xi-a$. Observe that, given a random uniform vector $\xi\in \Zp^D$ and $a\in \Zp^{D}$ the random vector $\xi-a$ is also uniformly distributed in the unit ball; therefore $f$ and $f-f_a$ have the same distribution.
It follows that for every $m\in \mathbb{N}$
\be \mathbb{E}\#\{\sigma_f=0\}\cap U_m=\mathbb{E}\#\{\sigma_{f-f_a}=0\}\cap U_m.\ee
On the other hand, by \eqref{eq:trivi}, $\psi_{f_a}\equiv a$, and $A\in B(0, p^{-m})=\phi^{-1}(U_m)$ is a zero of $\psi_{f_a}$ if and only if $\psi_f(A)=a.$ Therefore:
\be \mathbb{E}\#\{\sigma_f=0\}\cap U_m=\mathbb{E}\#\{\psi_f=a\}\cap B(0, p^{-m})=\mathbb{E}\#\{\psi_f|_{B(0, p^{-m})}=a\}.\ee
We use now Proposition \ref{Kac} and write:
    \begin{align}\label{eq:fs}
   \mathbb{E}\#\{\psi_f|_{B(0, p^{-m})}=a\}&=\mathbb{E}\#\{\psi_f|_{B(0, p^{-m})}=0\}\\
   &=\int_{\Zp^D}\mathbb{E}\#\{\psi_f|_{B(0, p^{-m})}=a\}da\\
   &=\mathbb{E}\int_{\Zp^D}\#\{\psi_f|_{B(0, p^{-m})}=a\}da\\
  \label{eq:fs} &=\mathbb{E}\int_{B(0, p^{-m})}|\det (J\psi_f(x))|_pdx.
\end{align}
Notice that as $x\to 0$ we have
\be\label{eq:esti} |\det(J\psi_f(x))|=|\det (J\psi_f(0))|+o(\|x\|),\ee 
where the implied constant is uniformly bounded because the coefficients of the random vector $f$ range in the unit ball.
Therefore, using \eqref{eq:fs} and Lemma \ref{lemma:ecco}, we can write for every $m\in \mathbb{N}$:
\begin{align}
\mathbb{E}\#\{\sigma_f=0\}&=N_m \mathbb{E}\#\{\sigma_f|_{B(0, p^{-m})}=0\}\\
&=N_m \mathbb{E}\int_{B(0, p^{-m})}|\det J\psi_f(x)|_p dx\\&= \mu(\Qp\G_{k,n})\frac{1}{p^{mn^2}}\mathbb{E}\int_{B(0, p^{-m})}|\det J\psi_f(x)|_p dx\quad \eqref{eq:Um}\\
&=\mu(\Qp\G_{k,n})\frac{1}{p^{mn^2}}\mathbb{E}\int_{B(0, p^{-m})}\left(|\det (J\psi_f(0))|+o(\|x\|)\right)dx\quad \eqref{eq:esti}\\
&=\lim_{m\to \infty}\left(\mu(\Qp\G_{k,n})\frac{1}{p^{mn^2}}\mathbb{E}\int_{B(0, p^{-m})}\left(|\det (J\psi_f(0))|+o(\|x\|)\right)dx\right)\\
&=\mu(\Qp\G_{k,n})\mathbb{E}|\det (J\psi_f(0))|
\end{align}
This concludes the proof.
\end{proof}
\subsection{The structure of the matrix $J\psi_f(0)$}\label{Structure}
The random matrix $J(\psi_f(0))$ is filled with random variables uniformly distributed in $\Zp$ and it has a special shape that we are going to compute.

Let us start with the case of points on the intersection of $n$ hypersurfaces in $\QP^n$ i.e. $\nu=n$. Using \eqref{eq:trivi}, for all $ 0\le j\le n$ we have: 
\begin{align}
\psi_{f_j}(A)&=f_j\left(y_0\begin{pmatrix}
1\\
t_1\\
\vdots\\
t_n
\end{pmatrix}\right)=\sum_{|\alpha|=d_j}\xi_\alpha^{(j)} y_0^{\alpha_0}t_1^{\alpha_1}\cdots t_n^{\alpha_n}.
\end{align}
Therefore $J\psi_f(0)$ has the following shape: 
$$J\psi_f(0)\begin{pmatrix}
\xi_{d_1-1,1,0,\ldots,0}^{(1)} & \dots & \xi_{d_1-1,0,\ldots,0,1}^{(1)}\\
\vdots & & \vdots\\
\xi_{d_n-1,1,0,\ldots,0}^{(n)}& \ldots & \xi_{d_n-1,0,\ldots,0,1}^{(n)}
\end{pmatrix}.$$

Let us now consider the case $n=3$ and $\nu=1$. In this case, for a matrix $A=\left(\begin{smallmatrix}t_1&t_2\\t_3&t_4\end{smallmatrix}\right)$, using \eqref{eq:trivi}, we can write:
\begin{align*}
 \psi_f(A)&=f\left(y_0\begin{pmatrix}
1\\
0\\
t_1\\
t_3
\end{pmatrix}+y_1 \begin{pmatrix}
0\\
1\\
t_2\\
t_4
\end{pmatrix}\right)\\
&=\sum_{|\alpha|=3}\xi_{\alpha}y_0^{\alpha_0}y_1^{\alpha_1}(y_0t_1+y_1t_2)^{\alpha_2}(y_0t_3+y_1t_4)^{\alpha_3}
\end{align*}
Threfore $J\psi_f(0)$ has the following shape: 
 \[J\psi_f(0)=\begin{pmatrix}
\xi_{2010} &0 &\xi_{2001}  &0 \\ 
\xi_{1110} &\xi_{2010}  &\xi_{1101}  &\xi_{2001} \\ 
\xi_{0210} &\xi_{1110} &\xi_{0201}  &\xi_{1101} \\ 
0 &\xi_{0210}  &0  &\xi_{0201}
\end{pmatrix}.
\]
A simple computation generalizing this shows that for the case of lines on an hypersurface of degree $2n-3$ we have (distinct entries are uniform random variables in the unit ball): $$J\psi_f(0)=\begin{pmatrix}
    \xi_1^{(1)}& 0& \dots & \xi_1^{(n-1)}& 0\\
    \xi_2^{(1)}& \xi_1^{(1)}& \dots & \xi_2^{(n-1)}& \xi_1^{(n-1)}\\
    \vdots & \xi_2^{(1)} & &\vdots & \xi_2^{(n-1)}\\
    \xi_{2n-3}^{(1)}& \vdots && \xi_{2n-3}^{(n-1)}& \vdots\\
    0& \xi_{2n-3}^{(1)} & \dots & 0& \xi_{2n-3}^{(n-1)}
    \end{pmatrix}.$$
For the general case of $k$-flats on $Z(f_1,\ldots,f_\nu)$:    
\[J\psi_f(0)=\begin{pmatrix}
M_1\\
\vdots\\
M_j\\
\vdots\\
M_\nu
\end{pmatrix},\]
where 
$M_j$ is the Jacobian of the map given by the restriction of the polynomial $f_j$ to the space $\{x_{k+1}=\ldots=x_n=0\}$. Using again \eqref{eq:trivi} we get: \be f_j\left(\begin{pmatrix}
    \mathbf{1}_{k+1} \\
    A
    \end{pmatrix}y\right)=\sum_\alpha \xi_\alpha y_0^{\alpha_0}\cdots y_k^{\alpha_k}(y_0t_0^{(1)}+\ldots+ y_kt_k^{(1)})^{\alpha_{k+1}}\cdots (y_0t_0^{(n-k)}+\ldots+ y_kt_k^{(n-k)})^{\alpha_n}\ee
Thus the row corresponding to $y_0^{u_0}\cdots y_k^{u_k}$ ($u_0+\ldots +u_k=d_j$) in the matrix $M_j$ is given as follows: 
\begin{align}(\beta_{u_0-1,u_1,\cdots,u_k,1,0,\cdots,0}^{(j)},\cdots,\beta_{u_0,u_1,\cdots,u_k-1,1,0,\ldots,0}^{(j)}  ),&\ldots\\
\ldots&(\beta_{u_0-1,u_1,\cdots,u_k,0,\cdots,0,1}^{(j)},\cdots,\beta_{u_0,u_1,\cdots,u_k-1,0,\cdots,0,1}^{(j)}) ),\end{align}
where $$\beta_{u_0,\ldots,u_i-1,\ldots,u_k,0\ldots,1,\ldots,0}^{(j)}= \begin{cases}
0 \quad \textrm{if } u_i=0\\
\xi_{u_0,\ldots,u_i-1,\ldots,u_k,0\ldots,1,\ldots,0} ^{(j)}\quad \textrm{if } u_i\ge 1
\end{cases}$$

\begin{remark}
Every entry that is not $0$ in the matrix $J\psi_f(0)$ is repeated exactly $k+1$ times and appears in different rows.
\end{remark}
\subsection{Proof of Theorem \ref{thm:complete}}
From the previous remark, we can write $\det(J\psi_f(0))$ as follows: $$\det(J\psi_f(0))=(\xi_{1,1}^{(1)})^{k+1}\cdot P_0 +  (\xi_{1,1}^{(1)})^k\cdot P_1 +\cdots + P_{k+1},$$  
where $P_0,\ldots P_{k+1}$  are homogeneous polynomials with less variables than $\det(J\psi_f(0))$. Then 
\begin{align*}
        \mathbb{E}\left\{|\det(J\psi_f(0))|_p\right\}&\ge \mathbb{P}(|\det(J\psi_f(0))|_p=1)\\&\ge \mathbb{P}(|\det(J\psi_f(0))|_p=1\,\big|\, |P_0|_p=1)\cdot \mathbb{P}(|P_0|_p=1)\\
    \end{align*}
Under the assumption that $|P_0|_p=1$, the equation $\det(J\psi_f(0))=0$ has at most $k+1$ solutions in $\mathbb{Z}/p\mathbb{Z}$. Thus, $$\mathbb{P}(|\det(J\psi_f(0)|_p=1\,\big|\, |P_0|_p=1)\ge (1-\tfrac{k+1}{p}).$$
We repeat now the same procedure again with $\mathbb{P}(|P_0|=1)$ since every variable in the expression of $P_0$ appears at most with degree $k+1$; this process gets repeated at most $n-k$ times, since the matrix $J\psi_f(0)$ is of size $(k+1)\times(n-k)$. Therefore,
\begin{equation}\label{eq:1}
    \mathbb{E}\{|\det(J\psi_f(0))|_p\}\ge (1-\tfrac{k+1}{p})^{n-k}.
\end{equation}
On the other hand we have: \be\label{eq:2}\mathbb{E}\{|\det(J\psi_f(0))|_p\}\leq1,\ee 
because $\det(J\psi_f(0))$ takes its values in $\Zp$. Using \eqref{eq:volGL} for $\la(\textrm{GL}_m(\Zp))$, we  see that
\be\label{eq:3}\lim_{p\rightarrow +\infty}\frac{\la(\textrm{GL}_{n+1}(\Zp))}{\la(\textrm{GL}_{k+1}(\Zp))\cdot \la(\textrm{GL}_{n-k}(\Zp))}=1.\ee
Hence, using Theorem \ref{thm:reduction}, \eqref{eq:1}, \eqref{eq:2} and \eqref{eq:3}, we conclude that
\be \lim_{p\to \infty}\mathbb{E}\#\{\textrm{$k$-flats on $Z(f_1, \ldots, f_\nu)\subset \QP^n$}\}=1.\ee
\subsection{Proof of Theorem \ref{thm:complete0}}
Notice from Section \ref{Structure} that the matrix $J\psi_f(0)$ is a random matrix with all entries which are random  independent variables uniformly distributed in $\Zp$. This is independent of the choice of the degrees $d_1,\ldots,d_n$. Therefore, the expected number of points on $Z(f_1,\ldots,f_n)$ when the degrees $d_1,\ldots,d_n$ are distinct is the same when $d_1=\ldots=d_n=1$. Thus this is exactly the number of points in the intersection of $n$ generic hyperplanes lying in $\QP^n$. Therefore, 
$$\mathbb{E}\#Z(f_1,\ldots,f_n)=1.$$
\begin{corollary}\label{coro:matrix}
Let $M_n$ a random matrix in $\Zp^{n\times n}$ whose entries are random independent variables uniformly distributed in $\Zp$. Then:
\be
\label{eq:detM}\mathbb{E}\{|\det(M_n)|_p\}=\frac{(p-1)p^n}{p^{n+1}-1}.\ee
\end{corollary}
\begin{proof}
This is a direct consequence of Theorem \ref{thm:complete0} and Theorem \ref{thm:reduction}:
\be 1=\mathbb{E}\# Z(f_1, \ldots, f_n)=\mu(\QP^n)\cdot \mathbb{E}\{|\det M_n|_p\}=\frac{p^{n+1}-1}{p^n(p-1)}\cdot \mathbb{E}\{|\det M_n|_p\},\ee
from which \eqref{eq:detM} follows.
\end{proof}
\subsection{Proof of Theorem \ref{thm:cubic}}
Applying Theorem \ref{thm:reduction} in the case $k=1$, $\nu=1$ and $d=3$ we get: \[\mathbb{E}\#\left\{\sigma_f=0\right\}=\frac{\la(\textrm{GL}_{4}(\mathbb{Z}_p))}{\la(\textrm{GL}_2(\Zp))^2}\cdot \mathbb{E}\left\{|\det(J\psi_f(0))|_p\right\}.
\]
By \eqref{eq:volGL} we have:
\be\label{eq:sv}\frac{\la(\textrm{GL}_{4}(\mathbb{Z}_p))}{\la(\textrm{GL}_2(\Zp))^2} =\frac{(p^3-1)(p^2+1)}{p^4(p-1)},\ee
it remains to compute $\mathbb{E}\{|\det(  J\psi_f(0)|_p\}$ where
$$J\psi_f(0)=\begin{pmatrix}
\xi_1 & 0 & \xi_4 & 0\\
\xi_2 & \xi_1 & \xi_5 & \xi_4\\
\xi_3& \xi_2& \xi_6 & \xi_5\\
0& \xi_3 & 0& \xi_6
\end{pmatrix}$$
and $\xi_1,\ldots ,\xi_6$ are random variables i.i.d uniformly distributed in $\Zp$. 

Because of Lemma \ref{lemma:integral}, we can compute the expectation of 
$\det(J\psi_f(0))= (\xi_1\xi_6-\xi_3\xi_4)^2-(\xi_1\xi_5-\xi_2\xi_4)(\xi_2\xi_6-\xi_3\xi_5)$ as:
\be\label{eq:sum1} \mathbb{E}\{|\det(  J\psi_f(0))|_p\}=\lim_{n\rightarrow \infty} \frac{1}{p^{6n}}\sum_{0\le \xi_1,\ldots ,\xi_6 \le p^n-1} |(\xi_1\xi_6-\xi_3\xi_4)^2-(\xi_1\xi_5-\xi_2\xi_4)(\xi_2\xi_6-\xi_3\xi_5)|_p.\ee

We are going now to introduce an alternative way of performing the summation in \eqref{eq:sum1}. Observe first that to every element $(\xi_1, \ldots, \xi_6)\in \Zp^6$ we can associate a matrix $M$ in $\Zp^{3\times 2}$:
\be M=\begin{pmatrix}
 \xi_1& \xi_4\\
 \xi_2 & \xi_5\\
 \xi_3& \xi_6
 \end{pmatrix},\ee
 as well as its reduction $\pi_n(M)\in (\mathbb{Z}/p^n\mathbb{P})^{3\times 2}$ modulo $p^n$. The determinant of $J\psi_f(0)$ is a function of the minors of $M$. Let $\mathcal{M}\subset (\mathbb{Z}/p^n\mathbb{Z})^{3\times 2}$ be the subset of matrices of maximal rank and let us consider the following map\footnote{Recall that $\mathbb{P}^2(\Zn)$ is defined by
 \be \mathbb{P}^2(\Zn):=\left((\Zn)^3\setminus (0,0,0)\right)/\sim,\ee 
 where $(x,y,z)\sim (x',y',z')$ if and only if there exists $\la$ invertible in $\Zn$ such that $(x,y,z)=\la(x',y',z')$. }: 
\begin{align*}
    h: & \,\,\,\,\,\mathcal{M}\subset (\mathbb{Z}/p^n\mathbb{Z})^{3\times 2}  \to  \mathbb{P}^2(\mathbb{Z}/p^n\mathbb{Z}) \\
    &\,\,\,\,\, (\xi_1, \ldots, \xi_6)\mapsto \left[\xi_1\xi_5-\xi_2\xi_4: \xi_1\xi_6-\xi_3\xi_4: \xi_2\xi_6-\xi_3\xi_5\right].
    \end{align*}
Observe that $h$ is surjective. Denote  by $Q$ the set of elements $\left[k_1:k_2:k_3\right]\in \mathbb{P}^2(\Zn)$ such that $k_2^2-k_1k_3=0$, and by $S:=h^{-1}(Q)$ the set of matrices in $\mathcal{M}$ for which the corresponding matrices in $\Zp^{4\times 4}$ have valuation of determinant greater or equal to $n$. Using this notation we see that the sum 
\be \sum_{0\le \xi_1,\ldots ,\xi_6 \le p^n-1} |(\xi_1\xi_6-\xi_3\xi_4)^2-(\xi_1\xi_5-\xi_2\xi_4)(\xi_2\xi_6-\xi_3\xi_5)|_p\ee
is bounded from below by
\be
 L_n:=\sum_{[k_1:k_2:k_3]\not\in Q} \#h^{-1}([k_1:k_2:k_3])\cdot|k_2^2-k_1k_3|_p,\ee
 and bounded from above by
 \be\label{eq:5}R_n:=\sum_{[k_1;k_2;k_3]\not\in Q}\#h^{-1}([k_1:k_2:k_3])\cdot |k_2^2-k_1k_3|_p+
  \frac{\#S}{p^n}.\ee
In particular, from \eqref{eq:sum1}, it will follow that:
\be\label{eq:sum2} \lim_{n\to \infty}\frac{L_n}{p^{6n}}\leq \mathbb{E}\{|\det(  J\psi_f(0))|_p\} \leq\lim_{n\to\infty}  \frac{R_n}{p^{6n}}.\ee
 In Lemma \ref{lemma:S} below we will prove that
\begin{equation}\label{eq:10}\lim_{n\rightarrow +\infty} \frac{\#S}{p^{7n}}=0,\end{equation}
from which we deduce that:
\be\label{eq:a} \mathbb{E}\{|\det(  J\psi_f(0))|_p\} =\lim_{n\to \infty}\frac{1}{p^{6n}}\sum_{[k_1:k_2:k_3]\not\in Q} \#h^{-1}([k_1:k_2:k_3])\cdot|k_2^2-k_1k_3|_p.\ee

Let us  compute the cardinality of the fibers of $h$. To this end, let $[k_1:k_2:k_3]\in \mathbb{P}^2(\mathbb{Z}/p^n\mathbb{Z})$. Without loss of generality we can assume that:
\be [k_1:k_2:k_3]= [p^{m_1}: p^{m_1}\lambda_2: p^{m_3}\lambda_3]\ee
with $m_1\leq m_2, m_3.$ In Lemma \ref{lemma:card} we will show that:
\be \label{eq:13} \#h^{-1}\left(\left[p^{m_1}:\la_2p^{m_2}:\la_3p^{m_3}\right]\right)=p^{4n-3-m_1}(p-1)(p+1)(p^{m_1+1}-1).\ee

Let now  $0\le m\le n-1$ be such that  $p^m$ is the largest power of $p$ that divides all $k_i$ in $\left[k_1:k_2:k_3\right]$; such elements $\left[k_1:k_2:k_3\right]$ with this property are in bijection with the unit\footnote{$[k_1:k_2:k_3]\in \mathbb{P}^2(\Zn)$ is called unit if at least one of the coordinates $k_i$ is invertible in $\Zn$} elements in $\mathbb{P}^2(\mathbb{Z}/p^{n-m}\mathbb{Z})$, via the map that sends $\left[k_1:k_2:k_3\right] $ to $\left[k_1/p^m: k_2/p^m: k_3/p^m\right]$. Then by \eqref{eq:13} we have 
\be\label{eq:ba}
L_n\le\sum_{m=0}^{n-1}\sum_{ \substack{
    [k_1;k_2;k_3] \textrm{ unit } \\\textrm{ in  }\mathbb{P}^2(\mathbb{Z}/p^{n-m}\mathbb{Z})}}
    \alpha_n\cdot (p^{m+1}-1)\cdot \frac{|k_2^2-k_1k_3|_p}{p^{3m}}\le R_n
\ee
where $\alpha_n:=p^{4n-3}(p-1)(p+1)$.

We will need a further step of reduction. For  $1\le k$ let us denote by $A(k)$ the number of unit elements $(k_1,k_2,k_3)\in \left(\mathbb{Z}/p^k\mathbb{Z}\right)^3$  such that $$k_2^2-k_1k_3=0\, (mod\ \ p^k),$$ and set $A(0)=1-1/p^3$. We claim that:
\begin{align} \lim_{n\to \infty}\frac{L_n}{p^{6n}}&=\lim_{n\to \infty}\frac{R_n}{p^{6n}}\\
\label{eq:limf}&=\lim_{n\to \infty}\frac{1}{p^{6n}}\alpha_n \cdot \sum_{m=0}^{n-1}\frac{p^{m+1}-1}{p^{3m}} \sum_{k=0}^{n-m-1} \frac{p^{3(n-m-k)}A(k)-p^{3(n-m-k-1)}A(k+1)}{p^{n-m-1}(p-1)p^k}.\end{align} 
In fact $p^{3(n-m-k)}A(k)$ represents the number of unit elements $(k_1,k_2,k_3)\in (\mathbb{Z}/p^{n-m}\mathbb{Z})^3$ such that $|k_2^2-k_1k_3|_p\le p^{-k}$. Therefore,
\[\frac{p^{3(n-m-k)}A(k)-p^{3(n-m-k-1)}A(k+1)}{p^{n-m-1}(p-1)}\]
represents the number of unit elements $[k_1:k_2:k_3]\in \mathbb{P}^2(\mathbb{Z}/p^{n-m}\mathbb{Z})$ such that $|k_2^2-k_1k_3|_p=p^{-k}$. Therefore \eqref{eq:limf} follows from \eqref{eq:ba}.

Using the formula for $A(k)$, given bellow in Lemma \ref{lemma:Ak}, we have:
\begin{align*}
  \sum_{k=1}^{n-m-1}  \frac{A(k)p^{3(n-m-k)}-A(k+1)p^{3(n-m-k-1)}}{p^{k}}&=& \sum_{k=1}^{n-m-1}p^{3(n-m-1)}(p^{-2k+3}-p^{-2k+1}\cdots\\
 && \cdots -p^{-2k+2}+p^{-2k})\\
  &=&p^{3(n-m-1)}(p-p^{-2(n-m-1)+1}-1\cdots\\
  &&\cdots+p^{-2(n-m-1)}),
\end{align*}
and, for $k=0$:
\[\frac{A(0)p^{3(n-m)}-A(1)p^{3(n-m-1)}}{p^0}=p^{3(n-m-1)}(p^3-p^2)
\]
In particular we have
\small
\begin{align*}
  &\alpha_n\sum_{m=0}^{n-1} \frac{p^{m+1}-1}{p^{2m}}\sum_{k=0}^{n-m-1} \frac{p^{3(n-m-k)}A(k)-p^{3(n-m-k-1)}A(k+1)}{p^{n+k-1}(p-1)}\\ =&\alpha_n\sum_{m=0}^{n-1}(p^{m+1}-1)p^{2n-5m-2}\left(\frac{p^3-p^2+p-1+p^{-2(n-m-1)}-p^{-2(n-m-1)+1}}{p-1}\right) 
  \\
  =&\alpha_n\sum_{m=0}^{n-1}(p^{m+1}-1)\left(p^{2n-2}(p^2+1)p^{-5m}-p^{-3m}\right)
  \\
  =&\alpha_np^{2n-2}(p^2+1)\sum_{m=0}^{n-1}\left(p^{-4m+1}-p^{-5m}\right)-\alpha_n\sum_{m=0}^{n-1}\left(p^{-2m+1}-p^{-3m}\right)\\
  =&\alpha_np^{2n-2}(p^2+1)\left(\frac{(p^{4n}-1)p^5}{p^{4n}(p^4-1)}-\frac{(p^{5n}-1)p^5}{p^{5n}(p^5-1)}\right)-\alpha_n\left(\frac{p^{2n}-1}{p^{2n-3}(p^2-1)}-\frac{p^{3n}-1}{p^{3n-3}(p^3-1)}\right).
\end{align*}
\normalsize
It remains now to put the pieces together. Combining \eqref{eq:a} with \eqref{eq:limf}, and substituting the value for $\alpha_n$, we get:
\begin{align*}
\mathbb{E}\left\{|\det(J\psi_f(0))|_p\right\}&=\lim_{n\rightarrow \infty} \frac{1}{p^{6n}}p^{4n-3}(p-1)(p+1)p^{2n-2}(p^2+1)\left(\frac{p^5}{p^4-1}-\frac{p^5}{p^5-1}\right)\\
&=\frac{p^4}{p^4+p^3+p^2+p+1}.
\end{align*}
Together with \eqref{eq:sv} this finally gives:
 \begin{equation*}
     \mathbb{E}\#\left\{\sigma_f=0\right\}=\frac{(p^3-1)(p^2+1)}{p^5-1}.
 \end{equation*}
 This concludes the proof of Theorem \ref{thm:cubic}. It remains to prove the lemmas that we have used in the proof.

\begin{lemma}\label{lemma:card}Using the above notations:
\be \#h^{-1}\left(\left[p^{m_1}:\la_2p^{m_2}:\la_3p^{m_3}\right]\right)=p^{4n-3-m_1}(p-1)(p+1)(p^{m_1+1}-1).\ee
\end{lemma}
\begin{proof}Now we need to compute $\#h^{-1}(\left[k_1:k_2:k_3\right])$. This depends only on $m_1$ if we consider $\left[k_1:k_2:k_3\right]=\left[p^{m_1}:\la_2p^{m_2}:\la_3p^{m_3}\right]$ as before. Indeed, let $g \in h^{-1}\left(\left[p^{m_1}:\la_2p^{m_2}:\la_3p^{m_3}\right]\right)$. By Proposition \ref{SVD}, $g$ has the following form: $$g=\begin{pmatrix}
   U\begin{pmatrix}
   p^u & 0\\
   0   & p^v
    \end{pmatrix}V\\
   x \hspace{35pt} y
\end{pmatrix} \in h^{-1}\left(\left[p^{m_1}:\la_2p^{m_2}:\la_3p^{m_3}\right]\right), $$ where $u+v=m_1$ and $U,V$ are invertible matrices in $(\Zn)^{2\times2}$. This is equivalent to
$$\begin{pmatrix}
   U\begin{pmatrix}
   p^u & 0\\
   0   & p^v
    \end{pmatrix}\\
   x' \hspace{35pt} y'
\end{pmatrix} \in h^{-1}\left(\left[p^{m_1}:\lambda_2p^{m_2}:\lambda_3p^{m_3}\right]\right),$$ where $(x',y')=(x,y)\cdot V^{-1}.$ Set $U=\left(\begin{smallmatrix}
   a & b\\
   c  & d
    \end{smallmatrix}\right)$, then we have \begin{equation}
        \begin{cases}
        ay'p^u-bx'p^v=\lambda_2p^{m_2}(ad-bc)\\
        cy'p^u-dx'p^v=\lambda_3p^{m_3}(ad-bc)
        \end{cases}.
    \end{equation}
The previous system is equivalent to $$\begin{pmatrix}
a & b\\
c & d
\end{pmatrix} \begin{pmatrix}
y'p^u\\
-x'p^v
\end{pmatrix}=\begin{pmatrix}
\lambda_2p^{m_2}(ad-bc)\\
\lambda_3p^{m_3}(ad-bc)
\end{pmatrix},$$
which gives:
\be\label{eq:11}\begin{pmatrix}
y'p^u\\
-x'p^v
\end{pmatrix}=\begin{pmatrix}
a & b\\
c & d
\end{pmatrix}^{-1}\begin{pmatrix}
\lambda_2p^{m_2}(ad-bc)\\
\lambda_3p^{m_3}(ad-bc)
\end{pmatrix}.\ee
Notice that $m_2,m_3\ge u,v$. For all $z\in \Zn$ the equation $p^vx'=p^vz$ has $p^v$ solutions in $\Zn$. In fact, this is equivalent to $p^{n-v}/x'-z$. Thus
\eqref{eq:11} has exactly $p^up^v=p^{m_1}$ solutions, which means that we have $p^{m_1}$ choices for $(x,y)$. Therefore, using Lemma \ref{Cardinal} and \eqref{eq:12} we finally get:
\begin{align}
    \#h^{-1}\left(\left[p^{m_1}:\la_2p^{m_2}:\la_3p^{m_3}\right]\right)&=p^{m_1}\cdot \#\left\{M\in (\Zn)^{2\times 2}\, \big|\, |\det(M)|_p=p^{-m_1}\right\}\\
    &=p^{m_1}\cdot p^{4n}\cdot \la \left(\left\{M\in (\Zp)^{2\times 2}\, \big|\, |\det(M)|_p=p^{-m_1}\right\}\right)\\
   &=p^{4n-3-m_1}(p-1)(p+1)(p^{m_1+1}-1).
\end{align}
\end{proof}

\begin{lemma}\label{lemma:S}Using the above notations we have:
\be\lim_{n\rightarrow \infty}\frac{\#S}{p^{7n}}=0.\ee
\end{lemma}
\begin{proof}
We claim first that the number of triples $(k_1,k_2,k_3)\in \left(\mathbb{Z}/p^n\mathbb{Z}\right)^3$ such that $k_2^2-k_1k_3=0$ is less than $2p^{2n}$.

Indeed, suppose first that $k_1=p^uk_1'\ne 0$ ($0\le u\le n-1$) with $k_1'$ invertible. Then $k_1k_3-k_2^2=0$ if and only if $p^uk_3=k_2^2\cdot k_1'^{-1}$. Fixing $k_1'$ and $k_2$, the last equation has either $p^u$ or $0$ solutions for $k_3$, depending on whether the valuation of $k_2$ is less than $u/2$ or not. Then the number of elements $k_2$ for which the equation has solutions is  $\le p^{n-u/2}$. Therefore, the number of triples $(k_1,k_2,k_3)$ such that $k_2^2=k_1k_3$ and $k_1\ne 0$ is less than or equal to\begin{equation*}
    \sum_{u=0}^{n-1} p^{n-u-1}(p-1)p^up^{n-u/2}=p^{2n-1}(p-1)\cdot\frac{p^{n/2}-1}{p^{n/2}}\cdot \frac{p^{n/2}}{p^{n/2}-1} \leq 2p^{2n-1}(p-1).
\end{equation*}

For the case when $k_1=0$ we must have also $k_2=0$, hence in this case there exist $p^n$ triples $(k_1,k_2,k_3)$ solving $k_2^2-k_1k_3=0$. In particular the number of solution of $k_2^2-k_1k_3=0$ in $(\mathbb{Z}/p^n\mathbb{Z})^3$ is less than $2p^{2n-1}(p-1)+p^n\leq 2p^{2n}$, as claimed. 

By \eqref{eq:13}, for every element $\left[k_1;k_2;k_3\right]\in \mathbb{P}^2(\mathbb{Z}/p^n\mathbb{Z}) $ we have $$\#h^{-1}(\left[k_1;k_2;k_3\right]) \le p^{4n-3}(p-1)(p+1)\frac{p^{n}-1}{p^{n-1}}.$$ Therefore,
 $$\#S\le p^{4n-3}(p-1)(p+1)\frac{p^{n}-1}{p^{n-1}}2p^{2n}\leq O(p^{6n}),$$
 which immediately implies:
 $$\lim_{n\rightarrow \infty}\frac{\#S}{p^{7n}}=0.$$
\end{proof}
 \begin{lemma}\label{lemma:Ak}
 Using the above notations,
 \be \forall\,  k\ge 1:\,  A(k)=p^{2k}-p^{2k-2}\ee
 \end{lemma}
 \begin{proof}
 Suppose $k_2$ is invertible, then  also $k_1$ and $k_3$ must be invertible. By fixing  specific values for $k_1$ and $k_2$, there exists only one possible value for $k_3$ for which $k_2^2-k_1k_3=0$. 
Therefore in this case we have $p^{2(k-1)}(p-1)^2$ elements $(k_1,k_2,k_3)$ satisfying the equation $k_2^2-k_1k_3=0$. Let us now suppose that $k_2$ is not invertible, then either $k_1$ or $k_3$ is invertible. Without loss of generality suppose $k_1$ is invertible, again $k_3$ can have only one unique value whenever you fix $k_2$ and $k_1$. Therefore in this case we have $2p^{2(k-1)}(p-1)$, and hence we get $A(k)=p^{2k}-p^{2k-2}$. 
 \end{proof}
\subsection{The intersection of two quadrics}
Let us consider now the problem of counting lines on a complete intersection of two quadrics ($\nu=2$ and $d_1=d_2=2$). Using Section \ref{Structure}, the matrix $J\psi_f(0)$ in this case is given as follow: 
\be\label{eq:Jacobi}J\psi_f(0)=\begin{pmatrix}
a_1 & 0& b_1 & 0 & c_1 &0\\
a_2 & a_1 & b_2 & b_1 & c_2& c_1\\
0 & a_2 & 0 & b_2 & 0 & c_2\\
a'_1 & 0& b'_1 & 0 & c'_1 &0\\
a'_2 & a'_1 & b'_2 & b'_1 & c'_2& c'_1\\
0 & a'_2 & 0 & b'_2 & 0 & c'_2
\end{pmatrix}.\ee
We can exchange some rows and columns, without changing the valuation of the determinant, and get the following matrix: 
$$\begin{pmatrix}
a_1 & b_1 & c_1 & 0& 0& 0 \\
a_2& b_2 & c_2 & a_1& b_1 & c_1\\
a'_1& b'_1& c'_1 & 0 & 0 &0\\
a'_2 & b'_2& c'_2 & a'_1 & b'_1 & c'_1\\
0& 0& 0 & a_2 & b_2 & c_2\\
0 & 0 & 0 & a'_2 & b'_2 & c'_2
\end{pmatrix}.$$
Notice that in this case $\det(J\psi_f(0))=k_1k_4-k_2k_3 $ where $k_1,\ldots, k_4$ are the minors of the matrix $$\begin{pmatrix}
a_1 & b_1 & c_1\\
a_2& b_2 & c_2\\
a'_1& b'_1& c'_1\\
a'_2 & b'_2& c'_2
\end{pmatrix}.$$ The special shape of this matrix allows us to apply the ideas of the proof of Theorem \ref{thm:cubic} and get the following result.
\begin{theorem}\label{thm:quadrics}The average number of $p$--adic lines on the intersection of two random quadrics in $\QP^4$ is $1$.
\end{theorem}
\begin{proof}
Applying Theorem \ref{thm:reduction} in the case $k=1$, $\nu=2$, and $n=4$ we get:
\be\mathbb{E}\#\{\sigma_f=0\}=\frac{\la(\textrm{GL}_5(\Zp))}{\la(\textrm{GL}_3(\Zp))\cdot \la(\textrm{GL}_2(\Zp))}\mathbb{E}\{|\det(J\psi_f(0)|_p\},\ee
where $J\psi_f(0)$ is given by \eqref{eq:Jacobi}.
By \eqref{eq:volGL} we have: 
\be\label{eq:volume}\frac{\la(\textrm{GL}_5(\Zp))}{\la(\textrm{GL}_3(\Zp))\cdot \la(\textrm{GL}_2(\Zp))}=\frac{(p^4-1)(p^5-1)}{p^6(p-1)(p^2-1)}.\ee
It remains to compute $\mathbb{E}\{|\det(J\psi_f(0))|_p\}$. By Lemma \ref{lemma:integral} we have: 
\be \mathbb{E}\{|\det(J\psi_f(0))|_p\}=\lim_{n\to+\infty}\frac{1}{p^{12n}}\sum_{0\le \xi_1\ldots \xi_{12}\le p^n-1} |\det(M(\xi_1, \ldots,\xi_{12}))|_p\ee
where \be\label{eq:17} M(\xi_1,\ldots,\xi_{12})=\begin{pmatrix}
\xi_1 & \xi_2 & \xi_3 & 0& 0& 0 \\
\xi_4& \xi_5 & \xi_6 & \xi_1& \xi_2 & \xi_3\\
\xi_7& \xi_8& \xi_9 & 0 & 0 &0\\
\xi_{10} & \xi_{11}& \xi_{12} & \xi_4 & \xi_5 & \xi_6\\
0& 0& 0 & \xi_7 & \xi_8 & \xi_9\\
0 & 0 & 0 & \xi_{10} & \xi_{11} & \xi_{12}
 \end{pmatrix}.\ee
 We are going to follow the same reasoning of the proof of Theorem \ref{thm:cubic}. Let us first consider the following surjective map: 
 \begin{align*}
    h: & \,\,\,\,\,\mathcal{M}\subset (\mathbb{Z}/p^n\mathbb{Z})^{4\times 3}  \to  \mathbb{P}^3(\mathbb{Z}/p^n\mathbb{Z}) \\
    &\,\,\,\,\, (\xi_1,\ldots, \xi_{12})\mapsto \left[k_1:k_2:k_3:k_4\right],
    \end{align*}  where $\mathcal{M}$ is the subset of matrices with full rank, and $k_1,\ldots,k_4$ are the minors of the following matrix:
\be \begin{pmatrix}
 \xi_1 &\xi_2 &\xi_3\\
    \xi_4 & \xi_5 &\xi_6\\
    \xi_7 & \xi_8 & \xi_9\\
    \xi_{10} & \xi_{11} & \xi_{12}
\end{pmatrix}.\ee
 Therefore the sum 
 \be \sum_{0\le \xi_1\ldots \xi_{12}\le p^n-1}|\det(M(\xi_1, \ldots,\xi_{12}))|_p\ee is bounded from below by:
 \[L_n':=\sum_{[k_1:\cdots:k_4]\in \mathbb{P}^2(\Zn)\setminus Q} \#h^{-1}([k_1:k_2:k_3:k_4])|k_1k_4-k_2k_3|_p\]
 and from above by: 
 \[R_n':=\sum_{[k_1:\cdots:k_4]\in \mathbb{P}^2(\Zn)\setminus Q} \#h^{-1}([k_1:k_2:k_3:k_4])|k_1k_4-k_2k_3|_p+ \frac{\#S}{p^n}.\]
Here $Q$ is the set of $[k_1:k_2:k_3:k_4]\in \mathbb{P}^3(\mathbb{Z}/p^{n}\mathbb{Z})$ such that $k_1k_4-k_2k_3=0$, and $S:=h^{-1}(Q)$ consists of matrices $ \mathcal{M}$ for which the corresponding matrices $M(\xi_1,\ldots,\xi_{12}) \in (\Zp)^{6\times 6}$ have valuation of determinant greater or equal $n$.

 We will show in Lemma \ref{lemma:SS} below that: 
 \be \lim_{n\to +\infty}\frac{\#S}{p^{13n}}=0,\ee
which implies:
\begin{align}\label{eq:ex}
    \mathbb{E}\{|\det(J\psi_f(0))|_p\}&=\lim_{n\to +\infty}\frac{L_n'}{p^{12n}}\\
    &=\lim_{n\to +\infty}\frac{R_n'}{p^{12n}}
\end{align}

Let us  compute the cardinality of the fibers of $h$. To this end, let $[k_1:k_2:k_3]\in \mathbb{P}^2(\mathbb{Z}/p^n\mathbb{Z})$. Without loss of generality we can assume that:
\be [k_1:k_2:k_3:k_4]= [p^{m_1}: p^{m_1}\lambda_2: p^{m_3}\lambda_3:p^{m_4}\la_4],\ee
with $m_1\leq m_2, m_3,m_4.$ In Lemma \ref{lemma:Fbr} we will show that: 
 \be h^{-1}(\left[p^{m_1}:\la_2p^{m_2}:\la_3p^{m_3}:\la_4p^{m_4}\right])= p^{9n-6-m_1}(p^{3}-1)(p^{m_1+1}-1)(p^{m_1+2}-1).\ee

Set $B(0)=1-\tfrac{1}{p^4}$ and for $k\ge 1$, let us denote by $B(k)$ the number of unit elements $(k_1,k_2,k_3,k_4)\in (\mathbb{Z}/p^k\mathbb{Z})^4$ such that 
    \be k_1k_4-k_2k_3=0\ee
 Reasoning as in the proof of Theorem \ref{thm:cubic} we get:
    \be L_n'\le \beta_n\cdot \sum_{m=0}^{n-1}\frac{(p^{m+1}-1)(p^{m+2}-1)}{p^{3m}(p-1)}\sum_{k=0}^{n-m-1}\frac{B(k)p^{4(n-m-k)}-B(k+1)p^{4(n-m-k-1)}}{p^{n-m-1+k}} \le R_n',\ee
    where $\beta_n=p^{9n-6}(p^3-1)$. 
    
    By \eqref{eq:ex} we get
    \be \mathbb{E}\{|\det(J\psi_f(0))|_p\}=\lim_{n\to \infty}\frac{T_n}{p^{12n}},\ee
    where:
    \be T_n:=\beta_n \sum_{m=0}^{n-1}\frac{(p^{m+1}-1)(p^{m+2}-1)}{p^{3m}(p-1)}\sum_{k=0}^{n-m-1}\frac{B(k)p^{4(n-m-k)}-B(k+1)p^{4(n-m-k-1)}}{p^{n-m-1+k}}.
    \ee
    
    Using the formula for $B(k)$ which is given in Lemma \ref{lemma:Bk} below we have, for $k\ge 1$:
$$\frac{B(k)p^{4(n-m-k)}-B(k+1)p^{4(n-m-k-1)}}{p^k}=p^{4(n-m-1)}(p^{-2k+4}-2p^{-2k+2}+p^{-2k}),$$
and
$$\sum_{k=1}^{n-m-1}\frac{B(k)p^{4(n-m-k)}-B(k+1)p^{4(n-m-k-1)}}{p^{k}}=p^{4(n-m-1)}(p^2-1-p^{-2n+2m+4}+p^{-2n+2m+2}).$$
For $k=0$ we have 
$$\frac{B(0)p^{4(n-m)}-B(1)p^{4(n-m-1)}}{p^0}=p^{4(n-m-1)}(p^4-p^3-p^2+p).$$
Putting these two together, we have:
$$\sum_{k=0}^{n-m-1}\frac{B(k)p^{4(n-m-k)}-B(k+1)p^{4(n-m-k-1)}}{p^{k}}=p^{4(n-m-1)}(p^3+1)(p-1)-p^{2(n-m-1)}(p^2-1).$$
In particular,
\begin{align}
    T_n&&=\beta_n \sum_{m=0}^{n-1}\frac{(p^{m+1}-1)(p^{m+2}-1)}{p^{3m}}(p^3+1)p^{3(n-m-1)}\cdots\mspace{275mu}\\
    &&\cdots-\beta_n \sum_{m=0}^{n-1}\frac{(p^{m+1}-1)(p^{m+2}-1)}{p^{3m}}(p+1)p^{n-m-1} \\
    &&=\beta_n(p^3+1)p^{3(n-1)}\sum_{m=0}^{n-1}\frac{p^{2m+3}-p^{m+2}-p^{m+1}+1}{p^{6m}}-\beta_n O(p^n)\mspace{205mu}\\
    &&=\beta_n(p^3+1)p^{3(n-1)}\sum_{m=0}^{n-1}(p^{-4m+3}-p^{-5m+2}-p^{-5m+1}+p^{-6m})-O(p^{10n})\mspace{125mu}\\
    &&=\beta_n(p^3+1)p^{3(n-1)}\left(p^3\frac{p^{4n}-1}{p^4-1}\frac{p^4}{p^{4n}}-p^2\frac{p^{5n}-1}{p^5-1}\frac{p^5}{p^{5n}}-p\frac{p^{5n}-1}{p^5-1}\frac{p^5}{p^{5n}}+\frac{p^{6n}-1}{p^{6}-1}\frac{p^6}{p^{6n}}\right)\cdots \mspace{15mu}\\
    &&\cdots-O(p^{10n})
\end{align}
 Therefore, 
 \begin{align}
\mathbb{E}\{|\det(J\psi_f(0))|_p\}=\lim_{n\to +\infty}\frac{T_n}{p^{12n}}=&\frac{p^6-1}{p^3}\left(\frac{p}{p^4-1}-\frac{p+1}{p^5-1}+\frac{1}{p^6-1}\right).
\end{align}
Multiplying this equation with \eqref{eq:volume} we get: 
$$\mathbb{E}\#\{\textrm{lines on the intersection of two quadrics in } \QP^4\}=1.$$
This concludes the proof of our theorem. \end{proof}

It remains to prove the lemmas that we have used.
\begin{lemma}\label{lemma:Fbr}
Using the above notations we have: 
\be h^{-1}(\left[p^{m_1}:\la_2p^{m_2}:\la_3p^{m_3}:\la_4p^{m_4}\right])= p^{9n-6-m_1}(p^{3}-1)(p^{m_1+1}-1)(p^{m_1+2}-1).\ee
\end{lemma}    
\begin{proof}Arguing as before, let us write a point in the fiber of $h$ as:
    $$g=\begin{pmatrix}
   U\begin{pmatrix}
   p^u & 0 & 0\\
   0   & p^v & 0\\
   0 & 0 & p^w
    \end{pmatrix}V\\
   x \hspace{29pt} y \hspace{29pt} z
\end{pmatrix} \in h^{-1}\left(\left[p^{m_1}:\la_2p^{m_2}:\la_3p^{m_3}: \la_4p^{m_4}\right]\right) $$ where $u+v+w=m_1$ and $U,V$ are invertible matrices in $(\Zn)^{3\times3}$. 

Set
$$U=\begin{pmatrix}
a&b&c\\
d&e&f\\
g&h&k
\end{pmatrix}.$$
Then 
\[gV^{-1}=\begin{pmatrix}
   ap^u & bp^v & cp^w\\
   dp^u   & ep^v & fp^w\\
   gp^u & hp^v & kp^w\\
   x' & y' & z'
    \end{pmatrix}\\
 \in h^{-1}\left(\left[p^{m_1}:\la_2p^{m_2}:\la_3p^{m_3}: \la_4p^{m_4}\right]\right)  \]
if and only if
\[\begin{cases}
x'p^{v+w}(bf-ce)-y'p^{u+w}(af-cd)+z'p^{u+v}(ae-bd)=\la_2p^{m_2}\det(U)\\
x'p^{v+w}(bk-ch)-y'p^{u+w}(ak-cg)+z'p^{u+v}(ah-bg)=\la_3p^{m_3}\det(U)\\
x'p^{v+w}(ek-fh)-y'p^{u+w}(dk-fg)+z'p^{u+v}(dh-eg)=\la_4p^{m_4}\det(U)
\end{cases}\]
In other words:
\[\begin{pmatrix}
p^{v+w}x'\\
-p^{u+w}y'\\
p^{u+v}z'
\end{pmatrix}=\mathrm{co}(U)^{-1}\begin{pmatrix}
\la_4p^{m_4}\\
\la_3p^{m_3}\\
\la_2p^{m_2}
\end{pmatrix}\det(U),\]
where $\mathrm{co}(U)$ is the matrix of cofactors of $U$.
This system has $p^{v+w}$ solutions for $x'$, $p^{u+w}$ for $y'$, and $p^{u+v}$ for $z'$. This means  that we have  $p^{2m_1}$ solutions for $(x,y,z)$. Therefore,
\begin{align}
        h^{-1}(\left[p^{m_1}:\la_2p^{m_2}:\la_3p^{m_3}:\la_4p^{m_4}\right]=&p^{2m_1}\cdot \#\left\{M\in (\Zn)^{3\times 3}\, \big|\, |\det(M)|_p=p^{-m_1}\right\}\\
        =&p^{2m_1}\cdot p^{9n}\cdot \la\left(\left\{M\in (\Zp)^{3\times 3}\, \big|\, |\det(M)|_p=p^{-m_1}\right\}\right)\\
  \label{eq:15}      =& p^{9n-6-m_1}(p^{3}-1)(p^{m_1+1}-1)(p^{m_1+2}-1).
    \end{align}
\end{proof}
\begin{lemma}\label{lemma:SS}
Using the notations above we have:
\be \lim_{n\to +\infty}\frac{\#S}{p^{13n}}=0\ee
\end{lemma}
\begin{proof}
Notice that the number of elements $(k_1,k_2,k_3,k_4)\in (\mathbb{Z}/p^n\mathbb{Z})^4$ such that $k_1k_4-k_2k_3=0$, viewing $(\mathbb{Z}/p^n\mathbb{Z})^4\simeq (\mathbb{Z}/p^n\mathbb{Z})^{2\times 2}$ represents the number of matrices in $(\Zn)^{2\times 2}$ with determinant $0$. Let us call this number $S_n$. It can be computed using \eqref{eq:Cardinal2} and \eqref{eq:12} as follows: 
\begin{align*}
    S_n&=p^{4n}-\sum_{m=0}^{n-1}\#\{ M\in (\Zn)^{2\times 2}\, | \, |\det(M)|_p=p^{-m}\}\\
    &=p^{4n}-p^{4n}\cdot \sum_{m=0}^{n-1} \la\left(\{ M\in (\Zp)^{2\times 2}\, | \, |\det(M)|_p=p^{-m}\}\right)\\ &=p^{4n}-p^{4n}\cdot \sum_{m=0}^{n-1} \frac{(p^2-1)(p^{m+1}-1)}{p^{2m+3}}\\
    &= p^{4n}-p^{3n-1}(p+1)(p^n-1)+p^{2n-1}(p^{2n}-1)\\
    &= p^{3n}+p^{3n-1}-p^{2n-1}.
\end{align*}
Let $\left[p^{m_1}:\la_1p^{m_2}:\la_3p^{m_3}:\la_4p^{m_4}\right]\in \mathbb{P}^3(\Zn)$ with $m_1\le m_2,m_3,m_4$: this element has $p^{n-m_1-1}(p-1)$ representatives in $(\Zn)^4$. Moreover, the number of $(k_1,k_2,k_3,k_4)\in (\Zn)^4$ such that $p^{m_1}$ divides all the coordinates is $p^{4(n-m_1)}$. Therefore the number of $\left[k_1:k_2:k_3:k_4\right]$ such that $k_1k_4=k_2k_3$ is less than or equal to: 
\[\sum_{m=\lfloor\frac{n}{2}\rfloor}^{n-1}\frac{p^{4(n-m)}}{p^{(n-m-1)}(p-1)} + \frac{p^{3n}+p^{3n-1}-p^{2n-1}}{p^{n-\lfloor\frac{n}{2}\rfloor}(p-1)}=O(p^{5n/2}).\]
On the other hand, by \eqref{eq:15} every $\left[k_1:k_2:k_3:k_4\right]$ has at most $p^{8n-5}(p^{3}-1)(p^{n}-1)(p^{n+1}-1)$ preimages under $h$. Thus, $\#S=O(p^{12n+n/2})$, and hence 
$$\lim_{n\rightarrow +\infty} \frac{\#S}{p^{13n}}=0.$$
\end{proof}
\begin{lemma}\label{lemma:Bk}
Using the above notations we have: 
 $$\forall \, k\ge 1:\, B(k)=p^{3k}+p^{3k-1}-p^{3k-2}-p^{3k-3}.$$
 \end{lemma}   
 \begin{proof}
 Suppose first that $k_1$ is invertible. 
    Then $k_4=k_3k_2k_1^{-1}$, and in this case we have $p^{3k-1}(p-1)$ elements satisfying the equation $k_1k_4=k_2k_3$. 
    
    Suppose now that $k_1$ is not invertible, then there are two cases: \begin{itemize}
        \item[-] $k_2$ invertible. In this case $k_1k_4k_2^{-1}=k_3,$ and the number of elements $(k_1,k_2,k_3,k_4)$ satisfying $k_1k_4=k_2k_3$ is $p^{3k-2}(p-1)$.
        \item[-] $k_2$ not invertible. Then either $k_3$ or $k_4$ is invertible. Suppose $k_3$ is invertible: then $k_2=k_1k_4k_3^{-1}$. In this case we have $p^{3k-2}(p-1)$ elements verifying $k_4=k_3k_2k_1^{-1}$. Suppose now $k_3$ is not invertible, then $k_4$ is invertible and we have $k_1=k_2k_3k_4^{-1}$. In this case the number of solutions to our equation is $p^{3k-1}(p-1)$.
        \end{itemize}
    Adding up the resulting numbers from the previous cases, we get $$B(k)=p^{3k}+p^{3k-1}-p^{3k-2}-p^{3k-3}.$$
 \end{proof}
\subsection{Lines on hypersurfaces: the limit as $n\to \infty$}
\begin{theorem}\label{thm:n}Let $f\in \Zp[x_0, \ldots, x_n]_{(2n-3)}$ be a random uniform polynomial. Then:
\be \limsup_{n\to \infty}\mathbb{E}\#\{\textrm{lines on $Z(f)\subset \QP^n$}\}\leq \frac{1}{\lambda(\textrm{GL}_2(\Zp))}.\ee
\end{theorem}
\begin{proof}
From Theorem \ref{thm:reduction} we know that:
\[\mathbb{E}\#\{\textrm{lines on $Z(f)\subset \QP^n$}\}=\frac{\la(\textrm{GL}_{n+1}(\Zp))}{\la(\textrm{GL}_2(\Zp)) \la(\textrm{GL}_{n-1}(\Zp))}\mathbb{E}\{|\det(J\psi_f(0))|_p\}.\]
Notice that $$\lim_{n\rightarrow +\infty} \frac{\la(\textrm{GL}_{n+1}(\Zp))}{\la(\textrm{GL}_{n-1}(\Zp))}=1$$
and also $\mathbb{E}\{|\det(J\psi_f(0))|_p\}\le 1$, then
\[\limsup_{n\to \infty}\mathbb{E}\#\{\textrm{lines on $Z(f)\subset \QP^n$}\}\leq \frac{1}{\lambda(\textrm{GL}_2(\Zp))}.\]
\end{proof}
 
\begin{remark}When $\nu=1$, it still makes sense to consider the previous limit over the numbers $n$ satisfying $\binom{d+k}{k}=(n-k)(k+1)$ for some $d$ and we get a similar result. 
However in  the general case of more equations, it is not clear which type of asymptotic to consider.
\end{remark}

\bibliographystyle{alpha}
\bibliography{flats}
\end{document}